\documentclass[reqno, 12pt]{amsart}
 \usepackage[a4paper, hmargin = 1.25 in, vmargin = 1.5 in]{geometry}
 \usepackage{amsmath, amssymb, amsfonts}
 \usepackage{amsthm}
 
 \newtheorem{theorem}{Theorem}[section]

 \newtheorem{definition}{Definition}[section]
 \numberwithin{equation}{section} 

\begin{document}

\begin{center}

\begin{title}
\title{\bf\Large{{Uniqueness and Stability of Solutions for a Coupled System of Nabla Fractional Difference Boundary Value Problems}}}
\end{title}

\vskip 0.25 in

\begin{author}
\author {Jagan Mohan Jonnalagadda\footnote[1]{Department of Mathematics, Birla Institute of Technology and Science Pilani, Hyderabad - 500078, Telangana, India. email: {j.jaganmohan@hotmail.com}}}
\end{author}

\end{center}

\vskip 0.25 in

\noindent{\bf Abstract:} In this article, we obtain sufficient conditions on existence, uniqueness and Ulam--Hyers stability of solutions for a coupled system of two-point nabla fractional difference boundary value problems, using Banach fixed point theorem and Urs's approach. Finally, we illustrate the applicability of established results through an example..
\vskip 0.2 in


\vskip 0.2 in


\vskip 0.2 in

\section{Introduction}

In this article, we consider the following coupled system of nabla fractional difference equations with conjugate boundary conditions
\begin{equation} \label{Coupled System}
\begin{cases}
\big{(}\nabla^{\alpha_{1}}_{\rho(a)}u_{1}\big{)}(t) + f_{1}(t, u_{1}(t), u_{2}(t)) = 0, \quad t \in \mathbb{N}^{b}_{a + 2}, \\ \big{(}\nabla^{\alpha_{2}}_{\rho(a)}u_{2}\big{)}(t) + f_{2}(t, u_{1}(t), u_{2}(t)) = 0, \quad t \in \mathbb{N}^{b}_{a + 2}, \\ u_{1}(a) = 0, ~ u_{1}(b) = 0, \\
u_{2}(a) = 0, ~ u_{2}(b) = 0,
\end{cases}
\end{equation}
where $a$, $b \in \mathbb{R}$ with $b - a \in \mathbb{N}_{2}$; $1 < \alpha_{1}, \alpha_{2} < 2$; $f_{1}$, $f_{2} : \mathbb{N}^{b}_{a + 1} \times \mathbb{R}^{2} \rightarrow \mathbb{R}$ and $\nabla^{\nu}_{\rho(a)}$ denotes the $\nu^{\text{th}}$-th order Riemann--Liouville type backward (nabla) difference operator where $\nu \in \{\alpha_{1}, \alpha_{2}\}$.

In 1940, Ulam \cite{Ul} posed a problem on the stability of functional equations and Hyers \cite{Hy} solved it in the next year for additive functions defined on Banach spaces. In 1978, Rassias \cite{Ra} provided a generalization of the Hyers theorem for linear mappings. Since then, several mathematicians investigated Ulam’s problem in different directions for various classes of functional equations \cite{Ju F, Ka}, differential equations \cite{Ju 1, Ju 2, Ju 3, Mi 1, Mi 2, Ru 1, Ru 2}, difference equations \cite{Ju 4, Ju 5, Ju 6, On}, fractional differential equations \cite{Al, Be, Du 1, Du 2, Kh 1, Kh 2, Ib, Wa}, and fractional difference equations \cite{Ch 1, Ja, Ch 2}.

In particular, Urs \cite{Urs} presented some Ulam--Hyers stability results for the coupled fixed point of a pair of contractive type operators on complete metric spaces. Motivated by this work, in this article, we study the Ulam--Hyers stability of \eqref{Coupled System}.

The present paper is organized as follows: Section 2 contains preliminaries on nabla fractional calculus. In sections 3 and 4, we establish sufficient conditions on uniqueness and Ulam--Hyers stability of solutions of the discrete fractional boundary value problem \eqref{Coupled System}, respectively. We present an example in section 4.

\section{Preliminaries}

\subsection{Nabla Fractional Calculus}

We use the following notations, definitions and known results of nabla fractional calculus throughout the article. Denote by $\mathbb{N}_{a} = \{a, a + 1, a + 2, \ldots\}$ and $\mathbb{N}^{b}_{a} = \{a, a + 1, a + 2, \ldots, b\}$ for any $a$, $b \in \mathbb{R}$ such that $b - a \in \mathbb{N}_{1}$.

\begin{definition}[See \cite{Bo}]
The backward jump operator $\rho : \mathbb{N}_{a} \rightarrow \mathbb{N}_{a}$ is defined by $$\rho(t) = \begin{cases} a, \hspace{0.43 in} t = a, \\ t - 1, \quad t \in \mathbb{N}_{a + 1}.\end{cases}$$
\end{definition}

\begin{definition}[See \cite{Ki, Po}] The Euler gamma function is defined by $$\Gamma (z) = \int_0^\infty t^{z - 1} e^{-t} dt, \quad \Re(z) > 0.$$ Using its reduction formula, the Euler gamma function can also be extended to the half-plane $\Re(z) \leq 0$ except for $z \in \{\ldots, -2, -1, 0\}$. 
\end{definition}

\begin{definition}[See \cite{Go}]
For $t \in \mathbb{R} \setminus \{\ldots, -2, -1, 0\}$ and $r \in \mathbb{R}$ such that $(t + r) \in \mathbb{R} \setminus \{\ldots, -2, -1, 0\}$, the generalized rising function is defined by
\begin{equation}
\nonumber t^{\overline{r}} = \frac{\Gamma(t + r)}{\Gamma(t)}.
\end{equation}
Also, if $t \in \{\ldots, -2, -1, 0\}$ and $r \in \mathbb{R}$ such that $(t + r) \in \mathbb{R} \setminus \{\ldots, -2, -1, 0\}$, then we use the convention that $t^{\overline{r}} = 0$.
\end{definition}

\begin{definition}[See \cite{Go}]
Let $\mu \in \mathbb{R} \setminus \{\ldots, -2, -1\}$. Define the $\mu^{th}$-order nabla fractional Taylor monomial by $$H_{\mu}(t, a) = \frac{(t - a)^{\overline{\mu}}}{\Gamma(\mu + 1)},$$ provided the right-hand side exists. Observe that $H_{\mu}(a, a) = 0$ and $H_{\mu}(t, a) = 0$ for all $\mu \in \{\ldots, -2, -1\}$ and $t \in \mathbb{N}_{a}$.
\end{definition}

\begin{definition}[See \cite{Bo}]
Let $u: \mathbb{N}_{a} \rightarrow \mathbb{R}$ and $N \in \mathbb{N}_1$. The first order backward (nabla) difference of $u$ is defined by $$\big{(}\nabla u\big{)}(t) = u(t) - u(t - 1), \quad t \in \mathbb{N}_{a + 1},$$ and the $N^{th}$-order nabla difference of $u$ is defined recursively by $$\big{(}{\nabla}^{N}u\big{)}(t) = \Big{(}\nabla\big{(}\nabla^{N - 1}u\big{)}\Big{)}(t), \quad t\in \mathbb{N}_{a + N}.$$
\end{definition}

\begin{definition}[See \cite{Go}]
Let $u: \mathbb{N}_{a + 1} \rightarrow \mathbb{R}$ and $N \in \mathbb{N}_1$. The $N^{\text{th}}$-order nabla sum of $u$ based at $a$ is given by
\begin{equation}
\nonumber \big{(}\nabla ^{-N}_{a}u\big{)}(t) = \sum^{t}_{s = a + 1}H_{N - 1}(t, \rho(s))u(s), \quad t \in \mathbb{N}_{a},
\end{equation}
where by convention $\big{(}\nabla ^{-N}_{a}u\big{)}(a) = 0$. We define $\big{(}\nabla ^{-0}_{a}u\big{)}(t) = u(t)$ for all $t \in \mathbb{N}_{a + 1}$.
\end{definition}

\begin{definition}[See \cite{Go}]
Let $u: \mathbb{N}_{a + 1} \rightarrow \mathbb{R}$ and $\nu > 0$. The $\nu^{\text{th}}$-order nabla sum of $u$ based at $a$ is given by
\begin{equation}
\nonumber \big{(}\nabla ^{-\nu}_{a}u\big{)}(t) = \sum^{t}_{s = a + 1}H_{\nu - 1}(t, \rho(s))u(s), \quad t \in \mathbb{N}_{a},
\end{equation}
where by convention $\big{(}\nabla ^{-\nu}_{a}u\big{)}(a) = 0$.
\end{definition}

\begin{definition}[See \cite{Go}]
Let $u: \mathbb{N}_{a + 1} \rightarrow \mathbb{R}$, $\nu > 0$ and choose $N \in \mathbb{N}_1$ such that $N - 1 < \nu \leq N$. The $\nu^{\text{th}}$-order nabla difference of $u$ is given by
\begin{equation}
\nonumber \big{(}\nabla ^{\nu}_{a}u\big{)}(t) = \Big{(}\nabla^N\big{(}\nabla_{a}^{-(N - \nu)}u\big{)}\Big{)}(t), \quad t\in\mathbb{N}_{a + N}.
\end{equation}
\end{definition}

\subsection{Boundary Value Problem}
Let $a$, $b \in \mathbb{R}$ with $b - a \in \mathbb{N}_{2}$. Assume $1 < \alpha < 2$ and $h : \mathbb{N}^{b}_{a + 1} \rightarrow \mathbb{R}$. Consider the boundary value problem
\begin{equation} \label{BVP 1}
\begin{cases}
\big{(}\nabla^{\alpha}_{\rho(a)}u\big{)}(t) + h(t) = 0, \quad t \in \mathbb{N}^{b}_{a + 2}, \\ u(a) = 0, ~ u(b) = 0.
\end{cases}
\end{equation}

Brackins \cite{Br}, Gholami et al. \cite{Gh} and the author \cite{Jo} have obtained the following expression for the unique solution of \eqref{BVP 1}, independently.

\begin{theorem} \cite{Br, Gh, Jo}
The nabla fractional boundary value problem \eqref{BVP 1} has the unique solution 
\begin{equation} \label{Sol}
u(t) = \sum^{b}_{s = a + 1}G(t, s)h(s), \quad t \in \mathbb{N}^{b}_{a},
\end{equation}
where
\begin{equation} \label{Jo Green}
G(t, s) = \frac{1}{\Gamma(\alpha)}\begin{cases}
\frac{(b - s + 1)^{\overline{\alpha - 1}}}{(b - a)^{\overline{\alpha - 1}}}(t - a)^{\overline{\alpha - 1}}, \hspace{97pt} t \in \mathbb{N}^{\rho(s)}_{a},\\
\frac{(b - s + 1)^{\overline{\alpha - 1}}}{(b - a)^{\overline{\alpha - 1}}}(t - a)^{\overline{\alpha - 1}} - (t - s + 1)^{\overline{\alpha - 1}}, \quad t \in \mathbb{N}^{b}_{s}.
\end{cases}
\end{equation}
\end{theorem}

\begin{theorem} \cite{Br} \label{Properties}
The Green’s function $G(t, s)$ defined in \eqref{Jo Green} satisfies the following properties:
\begin{enumerate}
\item $G(a, s) = G(b, s) = 0$ for all $s \in \mathbb{N}^{b}_{a + 1}$; 
\item $G(t, a + 1) = 0$ for all $t \in \mathbb{N}^{b}_{a}$; 
\item $G(t, s) > 0$ for all $(t, s) \in \mathbb{N}^{b - 1}_{a + 1} \times \mathbb{N}^{b}_{a + 2}$; 
\item $\max_{t \in \mathbb{N}^{b - 1}_{a + 1}}G(t, s) = G(s - 1, s)$ for all $s \in \mathbb{N}^{b}_{a + 2}$; 
\item $\sum^{b}_{s = a + 1}G(t, s) \leq \lambda$ for all $(t, s) \in \mathbb{N}^{b}_{a} \times \mathbb{N}^{b}_{a + 1}$, where $$\lambda = \left(\frac{b - a - 1}{\alpha \Gamma(\alpha + 1)}\right)\left(\frac{(\alpha - 1)(b - a) + 1}{\alpha}\right)^{\overline{\alpha - 1}}.$$
\end{enumerate}
\end{theorem}

\section{Existence \& Uniqueness of Solutions of \eqref{Coupled System}}

Let $X = \mathbb{R}^{b - a + 1}$ be the Banach space of all real $(b - a + 1)$-tuples equipped with the maximum norm $$\|u\|_{X} = \max_{t \in \mathbb{N}^{b}_{a}}|u(t)|.$$ Obviously, the product space $\big(X \times X, \| \cdot \|_{X \times X}\big)$ is also a Banach space with the norm $$\|(u_{1}, u_{2})\|_{X \times X} = \|u_{1}\|_{X} + \|u_{2}\|_{X}.$$ A closed ball with radius $R$ centered on the zero function in $X \times X$ is defined by $$\mathcal{B}_{R} = \{(u_{1}, u_{2}) \in X \times X : \|(u_{1}, u_{2})\|_{X \times X} \leq R \}.$$ Define the operator $T : X \times X \rightarrow X \times X$ by
\begin{equation}
T(u_{1}, u_{2})(t) = \begin{pmatrix}
T_{1}(u_{1}, u_{2})(t) \\
T_{2}(u_{1}, u_{2})(t)
\end{pmatrix}, \quad t \in \mathbb{N}^{b}_{a},
\end{equation}
where
\begin{equation} \label{A 1}
T_{1}(u_{1}, u_{2})(t)  = \sum^{b}_{s = a + 1}G_{1}(t, s)f_{1}(s, u_{1}(s), u_{2}(s)), \quad t \in \mathbb{N}^{b}_{a},
\end{equation}
and
\begin{equation} \label{A 2}
T_{2}(u_{1}, u_{2})(t)  = \sum^{b}_{s = a + 1}G_{2}(t, s)f_{2}(s, u_{1}(s), u_{2}(s)), \quad t \in \mathbb{N}^{b}_{a}.
\end{equation}
The Green's functions $G_{1}(t, s)$ and $G_{2}(t, s)$ are given by
\begin{equation} \label{Jo Green 1}
G_{1}(t, s) = \frac{1}{\Gamma(\alpha_{1})}\begin{cases}
\frac{(b - s + 1)^{\overline{\alpha_{1} - 1}}}{(b - a)^{\overline{\alpha_{1} - 1}}}(t - a)^{\overline{\alpha_{1} - 1}}, \hspace{97pt} t \in \mathbb{N}^{\rho(s)}_{a},\\
\frac{(b - s + 1)^{\overline{\alpha_{1} - 1}}}{(b - a)^{\overline{\alpha_{1} - 1}}}(t - a)^{\overline{\alpha_{1} - 1}} - (t - s + 1)^{\overline{\alpha_{1} - 1}}, \quad t \in \mathbb{N}^{b}_{s},
\end{cases}
\end{equation}
and 
\begin{equation} \label{Jo Green 2}
G_{2}(t, s) = \frac{1}{\Gamma(\alpha_{2})}\begin{cases}
\frac{(b - s + 1)^{\overline{\alpha_{2} - 1}}}{(b - a)^{\overline{\alpha_{2} - 1}}}(t - a)^{\overline{\alpha_{2} - 1}}, \hspace{97pt} t \in \mathbb{N}^{\rho(s)}_{a},\\
\frac{(b - s + 1)^{\overline{\alpha_{2} - 1}}}{(b - a)^{\overline{\alpha_{2} - 1}}}(t - a)^{\overline{\alpha_{2} - 1}} - (t - s + 1)^{\overline{\alpha_{2} - 1}}, \quad t \in \mathbb{N}^{b}_{s}.
\end{cases}
\end{equation}
From Theorem \ref{Properties}, we have $\sum^{b}_{s = a + 1}G_{1}(t, s) \leq \lambda_{1}$ for all $(t, s) \in \mathbb{N}^{b}_{a} \times \mathbb{N}^{b}_{a + 1}$, where $$\lambda_{1} = \left(\frac{b - a - 1}{\alpha_{1} \Gamma(\alpha_{1} + 1)}\right)\left(\frac{(\alpha_{1} - 1)(b - a) + 1}{\alpha_{1}}\right)^{\overline{\alpha_{1} - 1}},$$ and $\sum^{b}_{s = a + 1}G_{2}(t, s) \leq \lambda_{2}$ for all $(t, s) \in \mathbb{N}^{b}_{a} \times \mathbb{N}^{b}_{a + 1}$, where $$\lambda_{2} = \left(\frac{b - a - 1}{\alpha_{2} \Gamma(\alpha_{2} + 1)}\right)\left(\frac{(\alpha_{2} - 1)(b - a) + 1}{\alpha_{2}}\right)^{\overline{\alpha_{2} - 1}}.$$ Clearly, $(u_{1}, u_{2})$ is a fixed point of $T$ if and only if $(u_{1}, u_{2})$ is a solution of \eqref{Coupled System}. Assume
\begin{enumerate}
\item[(H1)] $f_{1}$, $f_{2} : \mathbb{N}^{b}_{a + 1} \times \mathbb{R}^{2} \rightarrow \mathbb{R}$ are continuous;
\item[(H2)] There exist constants $L_{1}$, $L_{2}$, $L_{3}$ and $L_{4}$ such that $$|f_{1}(t, u_{1}, u_{2}) - f_{2}(t, v_{1}, v_{2})| \leq L_{1}|u_{1} - v_{1}| + L_{2}|u_{2} - v_{2}|,$$ and $$|f_{2}(t, u_{1}, u_{2}) - f_{2}(t, v_{1}, v_{2})| \leq L_{3}|u_{1} - v_{1}| + L_{4}|u_{2} - v_{2}|,$$ for all $(t, u_{1}, u_{2})$, $(t, v_{1}, v_{2}) \in \mathbb{N}^{b}_{a + 1} \times \mathbb{R}^{2}$.
\item[(H3)] Take $$\max_{t \in \mathbb{N}^{b}_{a + 1}}\left|f_{1}(t, 0, 0)\right| = M_{1}, \quad \max_{t \in \mathbb{N}^{b}_{a + 1}}\left|f_{2}(t, 0, 0)\right| = M_{2}.$$
\item[(H4)] $L = \lambda_{1} (L_{1} + L_{2}) + \lambda_{2} (L_{3} + L_{4}) \in (0, 1)$.
\end{enumerate}

We apply Banach fixed point theorem to establish existence and uniqueness of solutions of \eqref{Coupled System}.

\begin{theorem} \label{Ban 2}
Assume (H1), (H2), (H3) and (H4) hold. If we choose $$R \geq \frac{(\lambda_{1} M_{1} + \lambda_{2} M_{2})}{1 - [\lambda_{1} (L_{1} + L_{2}) + \lambda_{2} (L_{3} + L_{4})]},$$ then the system \eqref{Coupled System} has a unique solution $(u_{1}, u_{2}) \in \mathcal{B}_{R}$.
\end{theorem}

\begin{proof}
Clearly, $T : \mathcal{B}_{R} \rightarrow X \times X$. First, we show that $T$ is a contraction mapping. To see this, let $(u_{1}, u_{2})$, $(v_{1}, v_{2}) \in \mathcal{B}_{R}$, $t \in \mathbb{N}^{b}_{a}$, and consider
\begin{align*}
& \big{|}T_{1}(u_{1}, u_{2})(t) - T_{1}(v_{1}, v_{2})(t)\big{|} \\ & \leq \sum^{b}_{s = a + 1}G_{1} (t, s)\big{|}f_{1}(s, u_{1}(s), u_{2}(s)) - f_{2}(s, v_{1}(s), v_{2}(s))\big{|} \\ & \leq \sum^{b}_{s = a + 1}G_{1} (t, s)\big{[}L_{1}|u_{1}(s) - v_{1}(s)| + L_{2}|u_{2}(s) - v_{2}(s)|\big{]} \\ & \leq \lambda_{1} \big{[}L_{1}\|u_{1} - v_{1}\|_{X} + L_{2}\|u_{2} - v_{2}\|_{X}\big{]},
\end{align*}
implying that
\begin{equation} \label{A 1 M}
\big{\|}T_{1}(u_{1}, u_{2}) - T_{1}(v_{1}, v_{2})\big{\|}_{X} \leq \lambda_{1} \big{[}L_{1}\|u_{1} - v_{1}\|_{X} + L_{2}\|u_{2} - v_{2}\|_{X}\big{]}.
\end{equation}
Similarly, we obtain
\begin{equation} \label{A 2 M}
\big{\|}T_{2}(u_{1}, u_{2}) - T_{2}(v_{1}, v_{2})\big{\|}_{X} \leq \lambda_{2} \big{[}L_{3}\|u_{1} - v_{1}\|_{X} + L_{4}\|u_{2} - v_{2}\|_{X}\big{]}.
\end{equation}
Thus, from \eqref{A 1 M} and \eqref{A 2 M}, we have
\begin{align*}
& \|T(u_{1}, u_{2}) - T(v_{1}, v_{2})\|_{X \times X} \\ & = \big{\|}T_{1}(u_{1}, u_{2}) - T_{1}(v_{1}, v_{2})\big{\|}_{X} + \big{\|}T_{2}(u_{1}, u_{2}) - T_{2}(v_{1}, v_{2})\big{\|}_{X} \\ & \leq \big{[}(\lambda_{1} L_{1} + \lambda_{2} L_{3})\|u_{1} - v_{1}\|_{X} + (\lambda_{1} L_{2} + \lambda_{2} L_{4})\|u_{2} - v_{2}\|_{X}\big{]} \\ & < L \big{[}(\|u_{1} - v_{1}\|_{X} + \|u_{2} - v_{2}\|_{X}\big{]} \\ & = L \|(u_{1}, u_{2}) - (v_{1}, v_{2})\|_{X \times X}.
\end{align*}
Since $L < 1$, $T$ is a contraction mapping with contraction constant $L$. Next, we show that
\begin{equation} \label{Con}
T(\partial \mathcal{B}_{R}) \subseteq \mathcal{B}_{R}.
\end{equation}
To see this, let $(u_{1}, u_{2}) \in \partial \mathcal{B}_{R}$, $t \in \mathbb{N}^{b}_{a}$, and consider
\begin{align*}
\big{|}T_{1}(u_{1}, u_{2})(t)\big{|} & \leq \sum^{b}_{s = a + 1}G_{1} (t, s)|f_{1}(s, u_{1}(s), u_{2}(s))| \\ & \leq \sum^{b}_{s = a + 1}G_{1} (t, s)|f_{1}(s, u_{1}(s), u_{2}(s)) - f_{1}(s, 0, 0)| \\ & + \sum^{b}_{s = a + 1}G_{1} (t, s)|f_{1}(s, 0, 0)| \\ & \leq \sum^{b}_{s = a + 1}G_{1} (t, s)\big{[}L_{1}|u_{1}(s)| + L_{2}|u_{2}(s)|\Big{]} + M_{1} \sum^{b}_{s = a + 1}G_{1} (t, s) \\ & \leq \lambda_{1} \big{[}(L_{1} + L_{2}) R + M_{1}\big{]},
\end{align*}
implying that
\begin{equation} \label{T 1 M}
\big{\|}T_{1}(u_{1}, u_{2})\big{\|}_{X} \leq \lambda_{1} \big{[}(L_{1} + L_{2}) R + M_{1}\big{]}.
\end{equation}
Similarly, we obtain
\begin{equation} \label{T 2 M}
\big{\|}T_{2}(u_{1}, u_{2})\big{\|}_{X} \leq \lambda_{2} \big{[}(L_{3} + L_{4}) R + M_{2}\big{]}.
\end{equation}
Thus, from \eqref{T 1 M} and \eqref{T 2 M}, we have
\begin{align*}
\|T(u_{1}, u_{2})\|_{X \times X} & = \big{\|}T_{1}(u_{1}, u_{2})\big{\|}_{X} + \big{\|}T_{2}(u_{1}, u_{2})\big{\|}_{X} \\ & \leq \big{[}\lambda_{1} (L_{1} + L_{2}) + \lambda_{2} (L_{3} + L_{4})\big{]}R + (\lambda_{1} M_{1} + \lambda_{2} M_{2}) \leq R,
\end{align*}
implying that \eqref{Con} holds.  Therefore, by Banach fixed point theorem, $T$ has a unique fixed point $(u_{1}, u_{2}) \in \mathcal{B}_{R}$. The proof is complete.
\end{proof}

\section{Stability of Solutions of \eqref{Coupled System}}

We use Urs's \cite{Urs} approach to establish Ulam--Hyers stability of solutions of \eqref{Coupled System}.

\begin{theorem} \cite{Urs}
Let $X$ be a Banach space and $T_{1}$, $T_{2} : X \times X \rightarrow X$ be two operators. Then, the operational equations system
\begin{equation} \label{System 1}
\begin{cases}
u_{1} = T_{1}(u_{1}, u_{2}), \\
u_{2} = T_{2}(u_{1}, u_{2}),
\end{cases}
\end{equation}
is said to be Ulam--Hyers stable if there exist $C_{1}$, $C_{2}$, $C_{3}$, $C_{4} > 0$ such that for each $\varepsilon_{1}$, $\varepsilon_{2} > 0$ and each solution-pair $(u_{1}^{*}, u_{2}^{*}) \in X \times X$ of the in-equations:
\begin{equation}
\begin{cases}
\|u_{1} - T_{1}(u_{1}, u_{2})\|_{X} \leq \varepsilon_{1}, \\
\|u_{2} - T_{2}(u_{1}, u_{2})\|_{X} \leq \varepsilon_{2},
\end{cases}
\end{equation}
there exists a solution $(v_{1}^{*}, v_{2}^{*}) \in X \times X$ of \eqref{System 1} such that
\begin{equation}
\begin{cases}
\|u_{1}^{*} - v_{1}^{*}\|_{X} \leq C_{1} \varepsilon_{1} + C_{2} \varepsilon_{2}, \\
\|u_{2}^{*} - v_{2}^{*}\|_{X} \leq C_{3} \varepsilon_{1} + C_{4} \varepsilon_{2}.
\end{cases}
\end{equation}
\end{theorem}

\begin{theorem} \cite{Urs}
Let $X$ be a Banach space, $T_{1}$, $T_{2} : X \times X \rightarrow X$ be two operators such that
\begin{equation} \label{System 2}
\begin{cases}
\|T_{1}(u_{1}, u_{2}) - T_{1}(v_{1}, v_{2})\|_{X} \leq k_{1} \|u_{1} - v_{1}\|_{X} + k_{2} \|u_{2} - v_{2}\|_{X}, \\
\|T_{2}(u_{1}, u_{2}) - T_{2}(v_{1}, v_{2})\|_{X} \leq k_{3} \|u_{1} - v_{1}\|_{X} + k_{4} \|u_{2} - v_{2}\|_{X},
\end{cases}
\end{equation}
for all $(u_{1}, u_{2})$, $(v_{1}, v_{2}) \in X \times X$. Suppose $$H = \begin{pmatrix}
k_{1} & k_{2} \\
k_{3} & k_{4}
\end{pmatrix}$$ converges to zero. Then, the operational equations system \eqref{System 1} is Ulam--Hyers stable.
\end{theorem}

Set
\begin{equation}
H =
\begin{pmatrix}
\lambda_{1} L_{1} & \lambda_{1} L_{2} \\
\lambda_{2} L_{3} & \lambda_{2} L_{4}
\end{pmatrix}.
\end{equation}

\begin{theorem} \label{Stab}
Assume (H1), (H2), (H3), and (H4) hold. Choose $$R \geq \frac{(\lambda_{1}  + \lambda_{2})M}{1 - [\lambda_{1} (L_{1} + L_{2}) + \lambda_{2} (L_{3} + L_{4})]}.$$ Further, assume the spectral radius of $H$ is less than one. Then, the solution of system \eqref{Coupled System} is Ulam--Hyers stable.
\end{theorem}

\begin{proof}
In view of Theorem \ref{Ban 2}, we have
\begin{equation}
\begin{cases}
\big{\|}T_{1}(u_{1}, u_{2}) - T_{1}(v_{1}, v_{2})\big{\|}_{X} \leq \lambda_{1} \big{[}L_{1}\|u_{1} - v_{1}\|_{X} + L_{2}\|u_{2} - v_{2}\|_{X}\big{]}, \\
\big{\|}T_{2}(u_{1}, u_{2}) - T_{2}(v_{1}, v_{2})\big{\|}_{X} \leq \lambda_{2} \big{[}L_{3}\|u_{1} - v_{1}\|_{X} + L_{4}\|u_{2} - v_{2}\|_{X}\big{]},
\end{cases}
\end{equation}
which implies that
\begin{equation}
\|T(u_{1}, u_{2}) - T(v_{1}, v_{2})\|_{X \times X} \leq H \begin{pmatrix}
\|u_{1} - v_{1}\|_{X} \\
\|v_{2} - v_{2}\|_{X}.
\end{pmatrix}
\end{equation}
Since the spectral radius of $H$ is less than one, the solution of \eqref{Coupled System} is Ulam--Hyers stable. The proof is complete.
\end{proof}

\section{Example}

Consider the following coupled system of two-point nabla fractional difference boundary value problems
\begin{equation} \label{Coupled System E}
\begin{cases}
\big{(}\nabla^{1.5}_{\rho(0)}u_{1}\big{)}(t) + (0.01)e^{-t}\left[1 + \tan^{-1}u_{1}(t) + \tan^{-1}u_{2}(t)\right] = 0, \quad t \in \mathbb{N}^{9}_{2}, \\ \big{(}\nabla^{1.5}_{\rho(0)}u_{2}\big{)}(t) + (0.02)\left[e^{-t} + \sin u_{1}(t) + \sin u_{2}(t)\right] = 0, \quad t \in \mathbb{N}^{9}_{2}, \\ u_{1}(0) = 0, ~ u_{1}(9) = 0, \\
u_{2}(0) = 0, ~ u_{2}(9) = 0.
\end{cases}
\end{equation}

Comparing \eqref{Coupled System} and \eqref{Coupled System E}, we have $a = 0$, $b = 9$, $\alpha_{1} = \alpha_{2} = 1.5$, $$f_{1}(t, u_{1}, u_{2}) = (0.01)e^{-t}\left[1 + \tan^{-1}u_{1} + \tan^{-1}u_{2}\right],$$ and $$f_{2}(t, u_{1}, u_{2}) = (0.02)\left[e^{-t} + \sin u_{1} + \sin u_{2}\right],$$ for all $(t, u_{1}, u_{2}) \in \mathbb{N}^{9}_{0} \times \mathbb{R}^2$. Clearly, $f_{1}$ and $f_{2}$ are continuous on $\mathbb{N}^{9}_{0} \times \mathbb{R}^2$. Next, $f_{1}$ and $f_{2}$ satisfy assumption (H2) with $L_{1} = 0.01$, $L_{2} = 0.01$, $L_{3} = 0.02$ and $L_{4} = 0.02$. We have, $$M_{1} = \max_{t \in \mathbb{N}^{9}_{1}}|f_{1}(t, 0, 0)| = \frac{0.01}{e},$$ $$M_{2} = \max_{t \in \mathbb{N}^{9}_{1}}|f_{2}(t, 0, 0)| = \frac{0.02}{e},$$ $$\lambda_{1} = \left(\frac{b - a - 1}{\alpha_{1} \Gamma(\alpha_{1} + 1)}\right)\left(\frac{(\alpha_{1} - 1)(b - a) + 1}{\alpha_{1}}\right)^{\overline{\alpha_{1} - 1}} \approx 7.4259,$$ and $$\lambda_{2} = \left(\frac{b - a - 1}{\alpha_{2} \Gamma(\alpha_{2} + 1)}\right)\left(\frac{(\alpha_{2} - 1)(b - a) + 1}{\alpha_{2}}\right)^{\overline{\alpha_{2} - 1}} \approx 7.4259.$$ Also, $$L = \lambda_{1} (L_{1} + L_{2}) + \lambda_{2} (L_{3} + L_{4}) \approx 0.4456 < 1,$$ implying that (H4) holds. Thus, by Theorem \ref{Ban 2}, the system \eqref{Coupled System E} has a unique solution $(u_{1}, u_{2}) \in \mathcal{B}_{R}$, where $$R \geq \frac{(\lambda_{1} M_{1} + \lambda_{2} M_{2})}{1 - [\lambda_{1} (L_{1} + L_{2}) + \lambda_{2} (L_{3} + L_{4})]} = 0.1479.$$ Further, $$H =
\begin{pmatrix}
\lambda_{1} L_{1} & \lambda_{1} L_{2} \\
\lambda_{2} L_{3} & \lambda_{2} L_{4}
\end{pmatrix} = \begin{pmatrix}
0.0743 & 0.0743 \\
0.1486 & 0.1486
\end{pmatrix}.$$ The spectral radius of $H$ is 0.0223 which is less than one. Hence, by Theorem \ref{Stab} the solution of \eqref{Coupled System} is Ulam--Hyers stable.

\end{document}